\documentclass[]{article}

\usepackage{amsmath, amssymb, amsthm}
\usepackage{enumitem, hyperref}

\usepackage{titlesec}

\titleformat{\subsection}[runin]         
{\normalfont\bfseries}           		 
{\thesubsection}                       
{.5em}                                  
{}[. \,]                     			 

\titlespacing*{\subsection}{0pt}{*1}{0pt} 

\usepackage{titling}
\newcommand{\AuthorInfo}{
	Martin de Borbon \\
	Department of Mathematics, Schofield Building \\
	Loughborough University, Epinal Way, LE11 3TU, UK \\
	\texttt{m.de-borbon@lboro.ac.uk}
}

\theoremstyle{plain}
\newtheorem{theorem}{Theorem}[section]
\newtheorem{lemma}[theorem]{Lemma}
\newtheorem{proposition}[theorem]{Proposition}
\newtheorem{corollary}[theorem]{Corollary}

\theoremstyle{definition}
\newtheorem{definition}[theorem]{Definition}

\newtheorem{remark}[theorem]{Remark}
\newtheorem{example}[theorem]{Example}

\newtheorem{claim}{Claim}

\newcommand{\C}{\mathbb{C}}

\renewcommand{\P}{\mathbb{P}}
\newcommand{\CP}{\mathbb{CP}}

\newcommand{\mO}{\mathcal{O}}

\newcommand{\tn}{\widetilde{\nabla}}

\newcommand{\bz}{\bar{z}}
\newcommand{\bi}{\bar{\imath}}
\newcommand{\bj}{\bar{\jmath}}
\newcommand{\bk}{\bar{k}}
\newcommand{\bl}{\bar{\ell}}

\newcommand{\bp}{\bar{p}}
\newcommand{\bq}{\bar{q}}

\DeclareMathOperator{\End}{End}
\newcommand{\ddb}{\p\op}

\newcommand{\newpar}{\vspace{.5em}}

\title{Local classification of K\"ahler metrics with constant holomorphic sectional curvature}
\author{Martin de Borbon}

\newcommand{\p}{\partial}
\newcommand{\op}{\bar{\partial}}

\begin{document}

\maketitle

\begin{abstract}
	We prove the local classification of K\"ahler metrics with constant holomorphic sectional curvature by exploiting the geometry of the bundle of \(1\)-jets of holomorphic functions.
\end{abstract}

\section{Introduction}

Let \(\varphi\) be a smooth real function on the unit ball \(B \subset \C^n\) such that \(\omega = i \ddb \varphi\) is a K\"ahler metric. Let \(J^1(\mO)\) be the bundle of \(1\)-jets of holomorphic functions on \(B\).
In this paper, we introduce two Hermitian forms \(H\) and \(K\) on \(J^1(\mO)\) naturally associated with \(\varphi\).
The Hermitian form \(H\) is positive definite, while \(K\) has signature \((1,n)\). We show that \(H\) (respectively \(K\)) is flat if and only if \(\omega\) has constant holomorphic sectional curvature \(\kappa = 2\) (respectively \(-2\)).
As an application, we prove the classical local classification of K\"ahler metrics with constant holomorphic sectional curvature (chsc).

\begin{theorem}\label{thm}
	\textup{(i)}
	If \(\omega\) has chsc \(\kappa=2\) then we can find holomorphic coordinates \(z = (z^1, \ldots, z^n)\) such that \(\omega = i \ddb \log (1+|z|^2)\).
	\textup{(ii)} If \(\omega\) has chsc \(\kappa=-2\) then we can find holomorphic coordinates such that \(\omega = -i \ddb \log (1-|z|^2)\).
\end{theorem}

This theorem was first proved by Bochner \cite{bochner} under the assumption that the K\"ahler potential \(\varphi\) is real analytic. For smooth metrics, the result follows from the Cartan-Ambrose-Hicks Theorem \cite{kn, binxu}. 
The proof presented in this paper can be extended to yield an analogous local classification result for chsc Kähler metrics with cone singularities along a smooth divisor, for cone angles in the range \((0,2\pi)\); the details of this extension will be given elsewhere.

\subsection{Acknowledgements}
I want to thank Dima Panov. Many of the ideas of this paper, including jet bundles, I learnt from him.

\section{General background}
In this section, we introduce background and notation on Hermitian forms and K\"ahler metrics which will be used throughout the paper.

\subsection{Hermitian forms}
Let \(V\) be a complex vector space. A Hermitian form is a map \(H: V \times V \to \C\) that is complex linear on the first argument and satisfies \(H(y,x) = \overline{H(x,y)}\). A Hermitian metric is a positive definite Hermitian form. The dual vector space \(V^*\) consists of all linear functions \(V \to \C\)
equipped with the obvious addition \((\alpha + \beta)(x) := \alpha(x) + \beta(x)\) and scalar multiplication \((\lambda \alpha)(x):= \lambda \cdot \alpha(x)\). The Hermitian form \(H\) is non-degenerate if the antilinear map \(b_H : V \to V^*\) given by \(b_H(x) = H(\cdot, x)\) is bijective. 

Let \(H\) be a non-degenerate Hermitian form. Then \(H\) induces a Hermitian form \(H^{\vee}\) on \(V^*\) given by
\[
H^{\vee}(\alpha, \beta) = H (y, x)  \,\, \text{ with } \,\, \alpha = b_H(x) \, \text{ and } \, \beta = b_H(y) \,.
\]
The reverse order is necessary to make \(H^{\vee}\) complex linear on the first argument.

We also denote by \(H = (H_{i\bj})\) the Hermitian matrix whose \((i,j)\)-entry is \(H_{i\bj} = H(e_i, e_j)\) for a given basis \(e_1, \ldots, e_n\) of \(V\). 
Let \(e^1, \ldots, e^n\) be the dual basis of \(V^*\) defined by \(e^i(e_j) = \delta^i_j\) and let \(H^{\vee} = (H^{i\bj})\) be the matrix whose \((i,j)\)-entry is  \(H^{i\bj} = H^{\vee}(e^i, e^j)\). It follows the definitions that
\[
H^{\vee} = (H^{-1})^t \,\,.
\]
If \(\dim V = 1\) we usually write \(h\) instead of \(H\) and \(h^{\vee} = h^{-1}\)\,.

\subsection{K\"ahler metrics}
Let \(g\) be a K\"ahler metric on a complex manifold \(M\). In local complex coordinates \((z^1, \ldots, z^n)\) the metric coefficients are given by
\[
g_{i \bj} = g \left( \p_i, \p_{\bj} \right) \,,
\] 
where \(\p_i = \p / \p z^i\) and \(\p_{\bj} = \p / \p \bz^j\). The K\"ahler form is \(\omega(u, v) = g(Iu, v)\) where \(I\) is the complex structure of \(M\). We also refer to \(\omega\) as the K\"ahler metric.
Locally, we can write \(\omega = i \ddb \varphi\) where \(\varphi\) is a smooth real function, called K\"ahler potential. In terms of a K\"ahler potential, the metric coefficients are given by
\[
g_{i\bj} = \p_i \p_{\bj} \varphi \,\,.
\]
The coefficients of the induced Hermitian metric on \(T^*M\) are given by
\[
g\left(dz^i , d\bz^j\right) = g^{i\bj} \,,
\]
where \((g^{i\bj})\) is the inverse transpose of \((g_{i\bj})\).


The K\"ahler metric \(g\) has constant holomorphic sectional curvature (chsc) equal to \(\kappa\) if the sectional curvature of every complex line in \(TM\) is equal to \(\kappa\). In local coordinates, this is equivalent to
\begin{equation}\label{eq:chsc}
	R_{i\bj k \bl} = \frac{\kappa}{2} \left(g_{i\bj} g_{k \bl} \,+\, g_{i\bl}g_{k\bj} \right) \,,
\end{equation}
where
\begin{equation}\label{eq:riemann}
R_{i\bj k \bl} = -\p_k \p_{\bl} g_{i\bj} \,+\, g^{p\bq} (\p_{\bl} g_{p\bj})  (\p_k g_{i\bq})	
\end{equation}
is the Riemann curvature tensor.

\subsection{Curvature of a Hermitian line bundle}
Let \(L\) be a holomorphic line bundle equipped with a Hermitian metric \(h\). 
The curvature form of \(h\) is the \((1,1)\)-form on the base manifold locally given by
\[
\omega_h = i \op \p \log h \,,
\]
where \(h\) is the norm squared of a non vanishing holomorphic section \(s\). 
In the local trivialization of \(L\) given by \(s\), the Chern connection of \(h\) is 
\[
\nabla = d + \p \log h 
\] 
and the curvature of \(\nabla\) is given by \(F_h = \op \p \log h\). 
The normalization \(\omega_h = i F_h\) makes the \((1,1)\)-form \(\omega_h\) real. 
If \(L^*\) is the dual line bundle equipped with the induced Hermitian metric \(h^{-1}\) then 
\[
\omega_{h^{-1}} = - \omega_h \,\,\,.
\]

\subsection{Fubini-Study metric}
The tautological line bundle \(\mO(-1)\) on \(\CP^n\) inherits a Hermitian metric \(h\) given by restriction of the Euclidean metric on \({\C}^{n+1}\). The Fubini-Study metric \(\omega_{FS}\) is minus the curvature form of \(h\). 
In local coordinates \(z =(z^1, \ldots, z^n) \mapsto [1, z^1, \ldots, z^n]\) the norm squared of the section \(s(z) = (1,z)\) of \(\mO(-1)\) is equal to \(1+|z|^2\). Using that  \(-\op \p = \ddb\), we have
\[
\omega_{FS} = i \ddb \log \left( 1 + |z|^2\right) \,.
\]
The Fubini-Study metric \(\omega_{FS}\) has chsc \(\kappa = 2\).

\subsection{Complex hyperbolic metric}
Let \(\C^{1,n}\) denote \(\C^{n+1}\) endowed with the Hermitian form of signature \((1,n)\) given by \(K = |dx^0|^2 - \sum_{i=1}^{n} |dx^i|^2\). Let \(B \subset \CP^n\) be the set of all complex lines \(L \subset \C^{1, n}\) such that \(K\) restricted to \(L\) is positive.
The line bundle \(\mO(-1)|_B\) inherits a Hermitian metric \(k\) given by restriction of \(K\). The complex hyperbolic metric \(\omega_{CH}\) is the curvature form of \(k\). In local coordinates \(z \mapsto [1, z]\) we have \(B = \{|z|< 1\}\) and
\[
\omega_{CH} = - i \ddb \log \left( 1 - |z|^2\right) \,.
\]
The complex hyperbolic metric \(\omega_{CH}\) has chsc \(\kappa = -2\).

\section{Relative Fubini-Study form}

In this section, we recall the construction of the relative Fubini-Study (and complex hyperbolic) form on the projectivization of a Hermitian vector bundle.
Throughout this section, we let \(M\) be a complex manifold and let \(E \to M\) be a holomorphic vector bundle of rank \(r+1\).


\subsection{Projective bundle}
The projectivization of \(E\) is the \(\CP^r\)-bundle 
\[
\P(E) \xrightarrow{\pi} M
\]  
whose fibre at \(p \in M\) consists of complex \(1\)-dimensional linear subspaces \(\ell \subset E_p\).

A good reason to introduce \(\P(E)\) is that it allows us to describe line subbundles of \(E\) in more geometric terms as sections of \(\P(E)\). 
In this paper, by line subbundle we mean vector subbundle of rank \(1\). 
We record this observation as the following evident result.

\begin{lemma}
	There is a natural correspondence between line subbundles \(L \subset E\) and holomorphic sections \(\sigma: M \to \P(E)\).
\end{lemma}

The tautological line bundle \(\mO_{\P(E)}(-1)\) is the line bundle on \(\P(E)\) whose fibre at \(\ell \subset E_p\) is \(\ell\). 

\subsection{Relative Fubini-Study form}
Suppose that \(E\) is endowed with a Hermitian metric \(H\). Then the line bundle \(\mO_{\P(E)}(-1)\) inherits a Hermitian metric \(h\) given by restriction of \(H\). 
The relative Fubini-Study form \(\eta_{FS}\) associated to \(H\) is the real \((1,1)\)-form on \(\P(E)\) defined as minus the curvature form of \(h\). 

More explicitly, let \(s_0, \ldots, s_r\) be a local holomorphic frame of \(E\) and let \(H_{i\bj} = H(s_i, s_j)\). In local coordinates  \((z,x) \mapsto [s_0(z) \,+\, x^i s_i(z)]\), where \(z\) are coordinates on the base manifold  and \(x = (x^1, \ldots, x^r)\) are coordinates on the \(\CP^r\) fibres, we have
\[
\eta_{FS} = i \ddb \log \left( H_{0\bar{0}} + \bar{x}^i H_{0\bi} + x^i H_{i\bar{0}} + x^i \bar{x}^j H_{i\bj} \right) \,.
\]

The restriction of \(\eta_{FS}\) to the fibres \(\cong \CP^r\) is positive and isometric to the Fubini-Study metric \(\omega_{FS}\). 
In general,  the form \(\eta_{FS}\) is not positive in directions transverse to the fibres of \(\P(E)\).
The relative Fubini-Study form has the following property.

\begin{lemma}\label{lem:pbfs}
	Let \(\eta_{FS}\) be the relative Fubini-Study form on \(\P(E)\) associated to the Hermitian metric \(H\).
	Let \(L \subset E\) be a line subbundle and let \(\sigma: M \to \P(E)\) be the section corresponding to \(L\). Then 
	\[
	\sigma^* \eta_{FS} = - \omega_h \,\,,
	\] 
	where \(\omega_h\) is the curvature form of the Hermitian metric \(h = H|_L\) .
\end{lemma}

\begin{proof}
	Let \(t\) be a non-zero section of \(\mO_{\P(E)}(-1)\) so that \(\eta_{FS} = - i \op \p \log |t|_H^2\). Then \(s = t \circ \sigma\) is a section of \(L\) and
	\[
	\sigma^* \eta_{FS} = -i \op \p  \log |t \circ \sigma|^2_H = - i \op \p \log |s|_h^2 = -\omega_h
	\]
	as wanted.
\end{proof}

\subsection{Relative complex hyperbolic form}

Suppose now that \(E\) has a Hermitian form \(K\) of signature \((1,r)\). Let \(U_{+} \subset \P(E)\) the set of lines where \(K > 0\). Then the line bundle \(\mO_{\P(E)}(-1)\) inherits a Hermitian metric \(k\) on the open set \(U_{+}\) given by restriction of \(K\).
The relative complex hyperbolic form \(\eta_{CH}\) associated to \(K\) is the \((1,1)\)-form on \(U_{+}\) defined as the curvature form of \(k\). 

The restriction of \(\eta_{CH}\) to the fibres of \(U_{+} \to M\) is the complex hyperbolic metric \(\omega_{CH}\) but in general, the form \(\eta_{CH}\) might not be positive in directions transverse to the fibres.

\begin{lemma}\label{lem:pbch}
	Let \(K\) be a Hermitian form on \(E\) of signature \((1,r)\)
	and let \(\eta_{CH}\) be the relative complex hyperbolic form on \(U_{+}\) associated to \(K\).
	Let \(L \subset E\) be a line subbundle such that \(K|_L > 0\) and let \(\tau: M \to U_{+}\) be the section corresponding to \(L\). Then 
	\[
	\tau^* \eta_{CH} = \omega_k \,\,,
	\] 
	where \(\omega_k\) is the curvature form of the Hermitian metric \(k = K|_L\) . 
\end{lemma}

\begin{proof}
	Same as Lemma \ref{lem:pbfs}.
\end{proof}


\section{Hermitian forms on \(1\)-jets}\label{sec:HK}
We now turn to the construction of Hermitian forms on the bundle of \(1\)-jets.
The construction resembles that of the Sasaki metric on the tangent bundle of a Riemannian manifold.

Throughout this section, we let \(L \to M\) be a holomorphic line bundle equipped with a Hermitian metric \(h\). We assume that the associated curvature form \(\omega = i \op \p \log h\) is positive, i.e., we assume that \(g(u, v) = \omega(u, Iv)\) is a K\"ahler metric on \(M\).
 
\subsection{Jet bundle}
Let \(J^1(L)\) be the bundle of \(1\)-jets of \(L\). The fibre of \(J^1(L)\) at \(p \in M\) consists of equivalence classes of germs of holomorphic sections of \(L\) around \(p\), where \(s \sim s'\) if \(s(p) = s'(p)\) and \(ds(p) = ds'(p)\). 
Write \(j(s)\) for the equivalence class of the germ \(s\).
The vector bundle \(J^1(L)\) has rank \(n+1\) and we have an exact sequence:
\begin{equation}\label{eq:jetseq}
	0 \to T^*M \otimes L \to J^1(L) \xrightarrow{\pi} L \to 0 \,,
\end{equation}
where \(\pi(j(s)) = s\).

\subsection{Hermitian metric on \(J^1(L)\)}
The connection \(\nabla = d + \p \log h\) gives a smooth splitting \(J^1(L) \cong L \oplus \big(T^*M \otimes L \big)\) of the sequence \eqref{eq:jetseq}. Using the Hermitian metrics \(h\) and \(g \otimes h\) on the summands (where \(g\) is the induced metric on \(T^*M\)) we introduce the following.

\begin{definition}\label{def:H}
	Let \(H\) be the Hermitian metric on \(J^1(L)\) given by
	\begin{equation}\label{eq:metricJ1}
	| j(s) |_H^2 = |s|_h^2 \,+\, |\nabla s|_{g\otimes h}^2 \,\,.
	\end{equation}
	We refer to \(H\) as the Hermitian metric on \(J^1(L)\) induced by \(h\). 
\end{definition}


\begin{remark}
	A closely related construction, in the context of K\"ahler-Einstein metrics with positive Ricci curvature is presented in \cite[\S 2]{tian} and \cite[\S 3]{chili}.
\end{remark}

\begin{example}
	Let \(M = \CP^1\) and let \(L = \mO(1)\) equipped with its standard Hermitian metric \(h\). Let \(x_1\), \(x_2\) be orthonormal linear coordinates on \(\C^2\), then \(j(x_1)\) and \(j(x_2)\) are pointwise linearly independent sections of \(J^1(\mO(1)) \cong \underline{\C}^2\). We claim that
	\[
	\begin{pmatrix}
	|j(x_1)|_H^2 & H \big(j(x_1), j(x_2)\big) \\
	H \big(j(x_2), j(x_1)\big) & |j(x_2)|_H^2 
	\end{pmatrix}
	\, = \,
	\begin{pmatrix}
	1 & 0 \\
	0 & 1
	\end{pmatrix} \,.
	\]
	
	To check, let \(U = \{[x_1, x_2] \in \CP^1 \,|\, x_1 \neq 0\}\) and take the local coordinate \(z \mapsto [1,z]\). Trivialize \(\mO(1)\) over \(U\) by the section \(s = x_1\) so \(h = (1+|z|^2)^{-1}\).
	Then \(\omega = h^2 i dz \wedge d\bar{z}\) from which we get \(|dz|^2_{g\otimes h} = h^{-1}\). Using that \(\nabla 1 = \p \log h = -\bz h dz\), we obtain
	\[
	|j(1)|_H^2 = |1|^2_h \,+\, |\nabla 1|^2_{g\otimes h} = h + |z|^2 h = 1 \,.
	\]
	Similarly, \(|j(z)|_H^2 = 1\) and \(H \big(j(x_1), j(x_2)\big)= 0\).
\end{example}

\subsection{Hermitian form of signature \((1,n)\) on \(J^1(L^*)\)}
The Hermitian metric \(h\) on \(L\) induces a Hermitian metric \(h^{-1}\) on the dual line bundle \(L^*\). The connection \(\nabla^{*} = d + \log h^{-1}\) gives a smooth splitting of the sequence
\begin{equation}\label{eq:jetseq2}
	0 \to T^*M \otimes L^* \to J^1(L^*) \to L^* \to 0 \,.
\end{equation}
Using the Hermitian metrics on the summands, we introduce the following.

\begin{definition}\label{def:K}
	The Hermitian form \(K\) on \(J^1(L^*)\) is given by
	\begin{equation}
		|j(t)|^2_K = |t|^2_{h^{-1}} \,-\, |\nabla^*t|_{g \otimes h^{-1}}^2 \,\,\,.
	\end{equation}
	We refer to \(K\) as the Hermitian form of signature \((1,n)\) on \(J^1(L^*)\) induced by \(h\).
\end{definition}

\begin{remark}
	A closely related construction, in the context of K\"ahler-Einstein metrics with negative Ricci curvature is presented in \cite[p. 122]{hitchin} for Riemann surfaces and in higher dimensions by \cite[\S 9]{simpson}.
\end{remark}

\begin{example}
	Let \(B\) be the unit disc \(\{[1,z] \in \CP^1 \,|\, |z| < 1\}\) and let \(L = \mO(-1)|_B\) equipped with the Hermitian metric \(h = 1 - |z|^2\) induced from the constant Hermitian form \(|dx_1|^2 - |dx_2|^2\) on \(\C^2\).
	Then \(\omega = -i \ddb \log h\) is the hyperbolic metric of curvature \(-2\). An easy calculation shows that
	\[
	\begin{pmatrix}
	|j(x_1)|_{K}^2 & K \big(j(x_1), j(x_2)\big) \\
	K \big(j(x_2), j(x_1)\big) & |j(x_2)|_{K}^2 
	\end{pmatrix}
	\, = \,
	\begin{pmatrix}
	1 & 0 \\
	0 & -1
	\end{pmatrix} \,.
	\]
\end{example}

\section{Quotient isometry property}
The main result of this section is Proposition \ref{prop:pullback}. As we shall show, it follows directly from the fact that the projection maps 
\[
\big( J^1(L), H \big) \to \big(L, h\big) \text{ and } \big( J^1(L^*), K \big) \to \big(L^*, h^{-1}\big) 
\] 
are fibrewise quotient isometries. 

\subsection{Linear algebra background}
Let \(V\) and \(W\) be finite dimensional complex vector spaces equipped with non-degenerate Hermitian forms \(H\) and \(K\). Let \(f: V \to W\) be a surjective \(\C\)-linear map. We say that \(f\) is a \emph{quotient isometry} if \(V = \ker f \oplus (\ker f)^{\perp}\) and
\[
\forall x, y \in (\ker f)^{\perp}: \,\, K \big(f(x), f(y)\big) = H (x, y) \,\,.
\]
\begin{lemma}\label{lem:quotientisometry}
	Let \(f: V \to W\) be a quotient isometry. Then \(f^*: W^* \to V^*\) is an injective isometry with respect to the induced metrics \(K^{\vee}\) and \(H^{\vee}\).
\end{lemma}

\begin{proof}
	Let \(\varphi\) be the restriction of \(f\) to \((\ker f)^{\perp}\), so \(\varphi: (\ker f)^{\perp} \to W\) is a bijective isometry. Using the Hermitian forms \(H\) and \(K\) to identify \(V^*\) with \(V\) and \(W^*\) with \(W\), we have that \(f^* = (0, \varphi^{-1})\). Since the inverse of an isometry is also an isometry, the lemma follows. 
\end{proof}

\subsection{Line subbundles}
Writing \(E = J^1(L)\) and taking duals in \eqref{eq:jetseq} we obtain an exact sequence
\begin{equation*}
0 \to L^* \xrightarrow{\pi^*} E^* \to TM \otimes L^* \to 0 \,.
\end{equation*}
For simplicity of notation, we omit \(\pi^*\) and write \(L^* \subset E^*\).
Similarly, writing \(E_{1,n} = J^1(L^*)\) and taking duals in \eqref{eq:jetseq2} we obtain
\begin{equation*}
0 \to L \to E_{1,n}^* \to TM \otimes L \to 0 \,,
\end{equation*}
giving us the line subbundle \(L \subset E_{1,n}^*\).

\begin{corollary}\label{cor:restriction}
	Let \(H\) and \(K\) be as in Definitions \ref{def:H} and \ref{def:K}. Let \(H^{\vee}\) and \(K^{\vee}\) be the induced Hermitian forms on \(E^*\) and \(E_{1,n}^*\). Then the following holds:
	\begin{enumerate}[label=\textup{(\roman*)}]
		\item the restriction of \(H^{\vee}\) to the line subbundle \(L^* \subset E^*\) is \(h^{-1}\)\,;
		\item the restriction of \(K^{\vee}\) to the line subbundle \(L \subset E_{1,n}^*\) is \(h\) .
	\end{enumerate}
\end{corollary}

\begin{proof}
	By definition of the Hermitian forms \(H\) and \(K\), the vector bundle maps \(E \to L\) and \(E_{1,n} \to L^*\) are fibrewise quotient isometries. Therefore, items (i) and (ii) follow from Lemma \ref{lem:quotientisometry}\,.
\end{proof}

\subsection{Main result}
The line subbundle \(L^* \subset E^*\) defines a holomorphic section
\begin{equation}
	\sigma: M \to \P(E^*) \,.
\end{equation}
Similarly, the line subbundle \(L \subset E_{1,n}^*\) defines a holomorphic section
\begin{equation}
	\tau: M \to U_{+} \,,
\end{equation}
where \(U_{+}\) is the subset of \(\P\big(E_{1,n}^*\big)\) consisting of lines on which \(K^{\vee} > 0\). 
The fact that \(K^{\vee}\) restricted to \(L\) is \(>0\) follows from Corollary \ref{cor:restriction} (ii).

\begin{proposition}\label{prop:pullback}
	Let \(\eta_{FS}\) be the relative Fubini-Study form on \(\P(E^*)\) associated to \(H^{\vee}\) and let \(\eta_{CH}\) be the relative complex hyperbolic form on \(\P\big(E_{1,n}^*\big)\) associated to \(K^{\vee}\). Then the following holds: \textup{(i)} \(\sigma^* \eta_{FS} = \omega\);
	\textup{(ii)} \(\tau^* \eta_{CH} = \omega\).
\end{proposition}

\begin{proof}
	By Corollary \ref{cor:restriction} (i) the restriction of \(H^{\vee}\) to \(L^*\) is \(h^{-1}\). By Lemma \ref{lem:pbfs}  \(\sigma^* \eta_{FS} = -\omega_{h^{-1}} = \omega\). The proof of item (ii) is similar and left to the reader.
\end{proof}

\section{Curvature} 

In this section, we compute the curvature of the Hermitian forms on \(J^1(L)\) and \(J^1(L^*)\) defined in Section \ref{sec:HK} and relate flatness to constant holomorphic sectional curvature.

\subsection{Chern connection}
Let \(E\) be a holomorphic vector bundle equipped with a non-degenerate Hermitian form \(H\). The Chern connection of \(H\) is the unique connection on \(E\) such that \(\nabla H = 0\) and \(\nabla^{0,1} = \op_E\). 

Let \(s^1, \ldots, s^r\) be a local holomorphic frame of \(E\) and let \(H^{i \bj} = H(s^i, s^j)\). Then we can write \(\nabla s^i = \theta^i_k s^k\) where \(\theta^i_k\) are \((1,0)\)-forms. The equation \(\p H^{i\bj} = \theta^i_k H^{k\bj}\) implies
\[
\theta = (\p H) \cdot H^{-1}
\]
where \(\theta\) is the matrix of \(1\)-forms whose \((i,j)\)-th entry is \(\theta^i_j\) and \(H = (H^{i\bj})\). The curvature of \(\nabla\) is the \((1,1)\)-form with values in \(\End E\) given by \(F = d^{\nabla} \circ \nabla\). In a local holomorphic frame as above, we have \(F s^i = \Omega^i_j s_j \) where \(\Omega^i_j = \op \theta^i_j\). In matrix notation,
\[
\Omega = \op \theta \,.
\]
The Hermitian form \(H\) is said to be flat if \(F \equiv 0\).

\subsection{Main result}
The main result of this section is the following.
\begin{proposition}\label{prop:curv}
	Let \(H\) and \(K\) be as in Definitions \ref{def:H} and \ref{def:K}.
	Then the following holds:
	\begin{enumerate}[label=\textup{(\roman*)}]
		\item the Hermitian metric \(H\) is flat if and only if \(g\) has chsc \(\kappa = 2\) ;
		\item the Hermitian form \(K\) is flat if and only if \(g\) has chsc \(\kappa = -2\) . 
	\end{enumerate}
\end{proposition}

\begin{proof}
	(i) Let \(\tn\) be the Chern connection of \(H\) on \(J^1(L)\). We use the tilde to distinguish \(\tn\) from the connection \(\nabla\) of \((L, h)\). We want to calculate the curvature of \(\tn\) at an arbitrary fixed point \(p \in M\). \newpar
	
	\noindent\textbf{Step 1: local frame.}
	Let \((z^1, \ldots, z^n)\) be normal coordinates for the K\"ahler metric \(g\) centred at \(p\). With this choice, the metric coefficients of \(g\) satisfy
	\begin{equation}
	g_{i\bj} (0) = \delta_{ij} \,\, \text{ and } \,\, \p_k g_{i\bj} (0) = \p_{\bk} g_{i\bj} (0) = 0 \,.
	\end{equation}
	
	It is standard and easy to show (c.f. \cite[Lemma A.2]{jmr}) that we can take a holomorphic section \(s\) of \(L\) near \(p\) such that \(h = |s|_h^2\) satisfies
	\begin{equation}\label{eq:h}
	h(0) = 1 \,,\, dh(0) = 0 \,, \textup{ and } \,\, (\p_i \p_j h) (0) = (\p_{\bi} \p_{\bj} h) (0) = 0 
	\end{equation}
	for all \(1 \leq i, j \leq n\).
	Trivialize \(L\) so that the section \(s\) corresponds to the constant function \(1\) and  let \(s^0, \ldots, s^n\) be the holomorphic frame of \(J^1(L)\) given by
	\begin{equation}
	s^0 = j(1)\,,\, s^1 = j(z^1) \,,\, \ldots \,,\, s^n = j(z^n) \,.
	\end{equation} 
	
	\noindent\textbf{Step 2: calculation of metric coefficients.}
	Write \(H^{a\bar{b}} = H(s^a, s^b)\) for \(0 \leq a, b \leq n\) and let \(H\)  be the \((n+1) \times (n+1)\) matrix 
	\[
	H = (H^{a\bar{b}}) \,.
	\]
	
	In our local trivialization of \(L\), we have
	\[
	H(j(u), j(v)) = h \cdot (u \bar{v}) \,+\, h \cdot \langle \nabla u , \nabla v \rangle_{g} \,\,,
	\]
	where \(u\) and \(v\) are holomorphic functions and \(\nabla u\) is the \(1\)-form 
	\[
	\nabla u = du + u \cdot \p \log h
	\] 
	(and similarly for \(\nabla v\)) while the product \(\langle \alpha, \beta \rangle_g\) of the \(1\)-forms \(\alpha = \alpha_p dz^p\) and \(\beta = \beta_q dz^q\) is given by
	\[
	\langle \alpha, \beta \rangle_g = g^{p\bq} \alpha_p \overline{\beta_q} \,.
	\]
	Taking \(u = v = 1\), we get
	\begin{equation}\label{eq:H00}
	H^{0\bar{0}} = h \,+\, h^{-1}  |\p h|^2_{g} \,\,,
	\end{equation}
	where \(|\p h|^2_{g} = g^{p\bq} (\p_p h) (\p_{\bq} h)\).
	Taking \(u = z^i\) and \(v = 1\),
	we get
	\begin{equation}\label{eq:Hi0}
	H^{i\bar{0}} = hz^i \,+\, h^{-1} z^i |\p h|^2_{g}  \,+\,  g^{i\bq} (\p_{\bq} h) \,\,.
	\end{equation}
	Taking \(u = z^i\) and \(v=z^j\), we get
	\begin{equation}\label{eq:Hij}
	H^{i\bj} = 
	h z^i \bz^j  \,+\,
	hg^{i\bj} \,+\,  
	z^i g^{p\bj} (\p_p h) \,+\,
	\bz^j g^{i\bq} (\p_{\bq}h) \,+\,
	z^i \bz^j h^{-1} |\p h|^2_{g}  \,\,.
	\end{equation}

	\begin{claim}\label{claim:H0}
		The matrix \(H\) is equal to the identity at the origin.
	\end{claim}
	\begin{proof}[Proof of Claim 1]
		Follows immediately from Equations \eqref{eq:H00}, \eqref{eq:Hi0}, and \eqref{eq:Hij}.
	\end{proof}

	\noindent\textbf{Step 3: derivatives of metric coefficients.}
	To calculate \(dH\) we begin by noticing that \(\omega = - i \ddb h\) at the origin. Indeed, since \(\omega = - i \ddb \log h\), we have
	\begin{equation}\label{eq:gh}
		g_{k\bl} = -\p_k \p_{\bl} \log h = -\p_k \big(h^{-1} \p_{\bl} h \big) 
		= h^{-2} (\p_k h) (\p_{\bl}h) - \p_k \p_{\bl}  h \,.
	\end{equation}
	Evaluating at \(0\) we obtain
	\begin{equation}\label{eq:secondorderh}
	(\p_k \p_{\bl} h) (0) = -\delta_{kl} \,.
	\end{equation}
	
	\begin{claim}\label{claim:dH0}
		The matrix \(dH\) vanishes at the origin.
	\end{claim}
	\begin{proof}[Proof of Claim 2  ]
		Consider the coefficient \(H^{a\bar{b}} = H^{i\bar{0}}\) for \(1 \leq i \leq n\).
		Taking the derivative of Equation \eqref{eq:Hi0} with respect to \(z^k\) and evaluating at \(0\) gives us
		\[
		\p_k (H^{i\bar{0}}) = \delta^i_k \,+\, g^{i\bq} (\p_k \p_{\bq}h) \,\,.
		\] 
		Using Equation \eqref{eq:secondorderh} we get \(g^{i\bq} (\p_k \p_{\bq}h) = - \delta^{iq} \delta_{kq} = - \delta^i_k\), hence \(\p_k (H^{i\bar{0}}) = 0\).
		Similarly, 
		\(\p_{\bl} (H^{i\bar{0}}) =  \p_{\bl}\p_{\bi} h\) vanishes by Equation \eqref{eq:h}.
		The proofs that \(dH^{0\bar{0}}\) and \(dH^{i\bj}\) vanish at the origin  are easier and left to the reader. 
	\end{proof}

To calculate \(\ddb H\) at \(0\) we begin by noticing that
\begin{equation}\label{eq:thirdorderh}
(\p_k \p_{\bl} \p_p h) (0) = (\p_k \p_{\bl} \p_{\bq} h) (0) = 0 \,,
\end{equation}
as follows by differentiating Equation \eqref{eq:gh} and evaluating at \(0\).

\begin{claim}\label{claim:ddbH}
	(i) If \(a = 0\) or \(b=0\) then \((\p_k \p_{\bl} H^{a\bar{b}})(0) = 0\). (ii) 
	\begin{equation}\label{eq:Hijkl}
	(\p_k \p_{\bl}H^{i\bj}) (0) = (\p_k \p_{\bl}g^{i\bj})(0) - \left(\delta^{ij} \delta_{k\ell} + \delta^i_{k} \delta^j_{\ell} \right) \,.
	\end{equation}
\end{claim}

\begin{proof}[Proof of Claim 3]
	(i) Differentiating \eqref{eq:H00} and evaluating at the origin, 
	\[
	\p_k \p_{\bl} H^{0\bar{0}} = \p_k \p_{\bl}h \,+\, g^{p \bq}  (\p_p \p_{\bl} h) (\p_k \p_{\bq} h) 
	= -\delta_{k\ell} + \delta_{k\ell} = 0 \,.
	\]
	Differentiating \eqref{eq:Hi0} and evaluating at the origin, 
	\[
	\p_k \p_{\bl} H^{i\bar{0}} = g^{i\bq} (\p_k \p_{\bl} \p_{\bq} h) = \p_k \p_{\bl} \p_{\bi} h = 0 \,,
	\]
	where the last equality follows by Equation \eqref{eq:thirdorderh}.
	(ii) Differentiating \eqref{eq:Hij} and evaluating at the origin, 
	\[
	\begin{aligned}
	\p_k \p_{\bl}H^{i\bj} &= \p_k \p_{\bl} (z^i \bz^j h) +  \p_k \p_{\bl} (hg^{i\bj}) +  \p_k \p_{\bl} \left(\bz^j g^{i\bp} (\p_{\bp}h) \,+\, z^i g^{p\bj} (\p_p h) \right) \\
	&= \delta^i_{k} \delta^j_{\ell} -  \delta_{k\ell}\delta^{ij} + \p_k \p_{\bl}g^{i\bj} -2 \delta^i_{k} \delta^j_{\ell} 
	\end{aligned}
	\]
	from which Equation \eqref{eq:Hijkl} follows.
\end{proof}

\noindent \textbf{Final step: calculation of curvature.}
Since \(H(0)\) is the identity and \(dH(0) = 0\), the curvature of \(\tn\) at \(0\) is given by \(F = \op \p H\). By Claim \ref{claim:ddbH} (i) the \((a,b)\)-th entry of \(F\) is zero if either \(a=0\) or \(b=0\).
By Equation \eqref{eq:Hijkl}, \(F  = 0\) if and only if for all \(i, j, k, \ell \in \{1, \ldots, n\}\) we have
\[
(\p_k \p_{\bl}g^{i\bj})(0) = \delta^{ij} \delta_{k\ell} + \delta^i_{k} \delta^j_{\ell}  \,.
\]
On the other hand,
\begin{equation*}
(\p_k \p_{\bl} g^{i\bj}) (0) = -(\p_k \p_{\bl} g_{j \bi}) (0) = R_{j\bi k \bl}(0) \,,
\end{equation*}
where \(R\) is the Riemann curvature tensor of \(g\); the reversal in the order of the \(i, j\) indices is due to the fact that \((g^{i\bj})\) is the inverse transpose of \((g_{i\bj})\).
We conclude that \(F = 0\) if and only if
\[
R_{i \bj k \bl} = g_{i \bj} g_{k\bl} + g_{k\bj} g_{i\bl} \,\,,
\]
or equivalently, if and only if \(g\) has chsc \(\kappa = 2\).
This finishes the proof of Proposition \ref{prop:curv} (i). 

The proof of Proposition \ref{prop:curv} (ii) is analogous, replacing \(h\) with \(h^{-1}\) and introducing negative signs as appropriate, so we will be brief. We trivialize \(J^1(L^*)\) near \(p\) with sections \(t^0 = j(s^*)\) and \(t^i = j(z^i s^*)\), where \(s\) and \(z^i\) are as in Step 1 and \(s^*(s) = 1\).
Let \(K\) be the matrix with entries \(K^{a \bar{b}} = K(t^a, t^b)\). Same as before, \(K(0) = \) identity, \(dK(0) = 0\), and \((\p_k \p_{\bl} K^{a\bar{b}})(0) = 0\) if \(a =0\) or \(b=0\).
Moreover,
\[
K^{i\bj} = 
h^{-1} z^i \bz^j  \,-\,
h^{-1}g^{i\bj} \,-\,  
z^i g^{p\bj} (\p_p h^{-1}) \,-\,
\bz^j g^{i\bq} (\p_{\bq}h^{-1}) \,+\,
O(|z|^4)
\]
and, at the origin,
\[
\begin{aligned}
\p_k \p_{\bl}K^{i\bj} 
&= \delta^i_{k} \delta^j_{\ell} -  \delta_{k\ell}\delta^{ij} - \p_k \p_{\bl}g^{i\bj} -2 \delta^i_{k} \delta^j_{\ell} \\
&= - R_{j\bi k \bl} - \big( g_{j\bi} g_{k\bl} + g_{k\bi} g_{j\bl} \big) \,.
\end{aligned} 
\]
We conclude that the curvature of \(K\) at \(p\) -which is equal to \(\op \p K\)- vanishes if and only if \(g\) has chsc \(\kappa = -2\).
\end{proof}

\section{Proof of Theorem \ref{thm}}
Let \(\omega = i \ddb \varphi\) be a K\"ahler metric on the unit ball \(B \subset \C^n\) that has chsc \(\kappa = 2\). The K\"ahler potential \(\varphi\) defines a Hermitian metric \(h= |1|_h^2 \) on the trivial line bundle \(L = \mO\) with \(h = \exp(-\varphi)\) that has curvature form \(\omega_h = i \op \p \log h = \omega\). 
Let \(H\) be the Hermitian metric on \(E = J^1(\mO)\) induced by \(h\) as in Definition \ref{def:H}. 

By Proposition \ref{prop:curv}(i) the Chern connection \(\tn\) of \(H\) is flat. 
By the Frobenius Theorem, we can find a frame of parallel sections
of \(E\) on \(B\), say
\[
e = (e_0, e_1, \ldots, e_n) \,.
\] 
Since \(\tn^{0,1} = \op_E\), the frame \(e\) is holomorphic. Since \(\tn H = 0\), the products \(H(e_i, e_j)\) are constant. Thus, after composing with a constant linear transformation, we can assume that the frame \(e\) is orthonormal.

Using the dual frame, we identify \(\P(E^*)\) with the trivial bundle \(B \times \CP^n\) and write \(\pi_2 : \P(E^*) \to \CP^n\) for the projection to the fibre factor. The relative Fubini-Study form on \(\P(E^*)\) associated to \(H^{\vee}\) is then given by
\[
\eta_{FS} = \pi_2^*(\omega_{FS}) \,,
\]
where \(\omega_{FS}\) is the Fubini-Study metric on \(\CP^n\). 

Let \(\sigma: B \to \P(E^*)\) be the section defined by \(L^* \subset E^*\) and let
\[
f = \pi_2 \circ \sigma \,.
\]
The function \(f: B \to \CP^n\) is clearly holomorphic and
\[
f^* \omega_{FS} = \sigma^* \big( \pi_2^* \omega_{FS} \big) = \sigma^* \eta_{FS} = \omega \,,
\]
where the last equality follows from Proposition \ref{prop:pullback} (i). This completes the proof of Theorem \ref{thm} (i). The proof of item (ii) is analogous, with the obvious modifications, and is therefore omitted.

\bibliographystyle{alpha}
\bibliography{refs}	

\vspace{2em}
\begin{flushright}
\AuthorInfo	
\end{flushright}

\end{document}